\documentclass[11pt, reqno]{amsart}
\usepackage[autostyle]{csquotes}
\usepackage{color}
\usepackage{amsmath}
\usepackage{latexsym,amssymb}
\usepackage[mathscr]{euscript}
\usepackage{stmaryrd}
\usepackage{bm}
\usepackage{enumerate}
\usepackage{setspace}
\newtheorem{thm}{Theorem}[section]
\newtheorem{proposition}{Proposition}[section]
\newtheorem{example}{\bf Example}[section]

\newtheorem{remark}{\bf Remark}[section]
\newtheorem{lemma}{\bf Lemma}[section]
\newtheorem{definition}{Definition}[section]
\numberwithin{equation}{section}
\newtheorem{assumption}{Assumption}[section]

\begin{document}
	
	\baselineskip=17pt
	
	\title[]
	{Nonzero-sum risk-sensitive continuous-time stochastic games with ergodic costs.}

	\author[M. K. Ghosh]{Mrinal K. Ghosh}
	\address{ Department of Mathematics\\
	 Indian Institute of Science\\
	  Bangalore-560012, India.}
	  \email{mkg@iisc.ac.in}
	
	\author[S. Golui]{Subrata Golui}
	\address{Department of Mathematics\\
		Indian Institute of Technology Guwahati\\
		Guwahati, Assam, India}
	\email{golui@iitg.ac.in}
	
	\author[C. Pal]{Chandan Pal}
	\address{Department of Mathematics\\
		Indian Institute of Technology Guwahati\\
		Guwahati, Assam, India}
	\email{cpal@iitg.ac.in}
	
	\author[S. Pradhan]{Somnath Pradhan}
	\address{Department of Mathematics\\
	 Indian Institute of Science Education and Research\\
	  Pune, Maharashtra-411008, India}
	  \email{somnath@iiserpune.ac.in}

	
	\date{}
	
	\begin{abstract}
		\vspace{2mm}
		\noindent We study nonzero-sum stochastic games for continuous time Markov decision processes on a denumerable state space with risk-sensitive ergodic cost criterion. Transition rates and cost rates are allowed to be unbounded. Under a Lyapunov type stability assumption, we show that the corresponding system of coupled HJB equations admits a solution which leads to the existence of a Nash equilibrium in stationary strategies. We establish this using an approach involving principal eigenvalues associated with the HJB equations. Furthermore, exploiting appropriate stochastic representation of principal eigenfunctions, we completely characterize Nash equilibria in the space of stationary Markov strategies.
		\vspace{2mm}
		
		\noindent
		{\bf Keywords:}
		Nonzero-sum game, risk-sensitive ergodic cost criterion,  stationary strategies, coupled HJB equations, Fan's fixed point theorem, Nash equilibrium.
		
	\end{abstract}
	
	\maketitle

\section{INTRODUCTION}
 We consider a nonzero-sum stochastic game on the infinite time horizon for continuous time Markov decision processes (CTMDPs) on a denumerable state space. The performance evaluation criterion is exponential of integral cost which addresses the decision makers (i.e., players) attitude towards risk. In other words we address the problem of nonzero-sum risk sensitive stochastic games involving continuous time Markov decision processes. In the literature of stochastic games involving continuous time Markov decision processes, one usually studies the integral of the cost \cite{GH1}, \cite{GH2}, \cite{GH3}  which is the so called risk-neutral situation. In the exponential of integral cost, the evaluation criterion is multiplicative as opposed to the additive nature of evaluation criterion in the integral of cost case. This difference makes the risk sensitive case significantly different from its risk neutral counterpart. The study of risk sensitive criterion was first introduced in \cite{B2}; see \cite{W3} and the references therein. This criterion  is studied extensively in the context of MDP both in  discrete and continuous times; see, for instance \cite{CF}, \cite{MS1}, \cite{MS3},  \cite{FH1}, \cite{GL}, \cite{GLZ},  \cite{PP}, \cite{Z1}, and the references therein. The corresponding results for stochastic (dynamic) games are limited. Notable exceptions are  \cite{BG1}, \cite{BG2}, \cite{GKP}. In discrete time and discrete state space the risk-sensitive zero-sum stochastic games with  bounded cost and transition rates have been studied by Basu and Ghosh \cite{BG2} and nonzero-sum games in \cite{BG1}. For CTDMPs, zero-sum stochastic games with risk-sensitive costs for bounded cost and bounded transition rates have been studied in \cite{GKP}. One can see \cite{GP1}, \cite{WC2}, and the references therein for finite horizon risk-sensitive nonzero-sum games for CTMDPs. Recently risk sensitive continuous time Markov decision processes have been studied in \cite{BP}, \cite{GS}, \cite{KP1}, \cite{KP2}. In this present paper we extend the results of the above four papers to  nonzero-sum stochastic games. Using principal eigenvalue approach, under a Lyapunov type stability assumption, we have  shown that the corresponding system of coupled HJB equations admits a solution which in turn leads to the existence of Nash equilibrium in stationary strategies. Also, exploiting the stochastic representation of principal eigenfunction we  completely characterize all possible Nash equilibria in the space of stationary Markov strategies.
 The main motivation for studying this kind of games arises from their applications to many interesting problems, such as  controlled birth-and-death systems, telecommunication and queueing systems in which the transition and cost rates may be both unbounded.

Our main contribution in this paper is the following. We establish the existence and characterization of Nash equilibria under a blanket Lyapunov type stability assumption. To be more specific, we study ergodic nonzero sum risk-sensitive stochastic games for CTMDPs having the following features: (a) the transition and the cost rates may be unbounded (b) state space is countable (c) at any state of the system the space of admissible actions is compact (d) the strategies may be history dependent. To our knowledge, these results are new in the  literature of ergodic non-zero sum risk-sensitive games for CTMDPs. Similar risk-sensitive game problems for discrete time Markov decision processes have been studied under small costs and geometric ergodicity assumption in \cite{BG2}.

The rest of this paper is organized as follows: Section 2 deals with the problem description and preliminaries. The ergodic cost criterion is analyzed in Section 3. Under a Lyapunov type stability assumption(s), we first establish the existence of a solution to the corresponding coupled Hamilton-Jacobi-Bellman (HJB) equations. This in turn leads to the existence of a Nash equilibrium in stationary strategies (see Theorem \ref{theo 4.1}). In Section 4, we present an illustrative example.

\section{The game model}
For the sake of notational simplicity we treat two player game. The $N$-player game for $N\geq 3$, is analogous. The continuous-time two-person nonzero-sum stochastic game model which consists of the following elements
\begin{equation}
\{S, U_1,U_2, (U_1(i)\subset U_1, U_2(i)\subset U_2,i\in S),\bar{\pi}_{ij}( u_1,u_2),\bar{c}_1(i, u_1,u_2), \bar{c}_2(i, u_1,u_2)\}, \label{eq 2.1}
\end{equation}
where each component is described below:
\begin{itemize}
\item  $S$, called the state space, is assumed to be the set of all positive integers endowed with the discrete topology, i.e. $S=:\{1,2,\cdots\}$. 	
\item   $U_1$ and $U_2$ are the action sets for players 1 and 2, respectively. The action spaces $U_1$ and $U_2$ are assumed to be Borel spaces with the Borel $\sigma$-algebras $\mathcal{B}(U_1)$ and $\mathcal{B}(U_2)$, respectively.
\item  For each $i\in S$, $U_1(i)\in \mathcal{B}(U_1)$ and $U_2(i)\in \mathcal{B}(U_2)$ denote the sets of admissible actions for players 1 and 2 in state $i$, respectively.  Let $K:=\{(i, u_1,u_2)|i\in S, u_1\in U_1(i), u_2\in U_2(i)\}$, which is a Borel subset of $S\times U_1\times U_2$.\\
Throughout this paper, we assume that\\
\noindent {\bf (A1)(a)} For each $i\in S$, the admissible action spaces $U_k(i), k=1,2$, are nonempty and compact subsets of $U_k$.
	\item The transition rates $\bar{\pi}_{ij}(u_1,u_2) , (u_1,u_2)\in U_1(i)\times U_2(i), i,j \in S$, satisfy the condition  $\bar{\pi}_{ij}(u_1,u_2) \geq 0$  for all $i \neq j, \ (u_1,u_2)\in U_1(i)\times U_2(i)$. Also, we assume that: \\
\noindent {\bf (A1)(b)}
 The transition rates  $\bar{\pi}_{ij}(u_1,u_2)$ are conservative, i.e.,  $$ \sum_{j \in S}\bar{\pi}_{ij}(u_1,u_2)=0~ \mbox{for} ~i \in S ~\mbox{and}~ (u_1,u_2)\in U_1(i)\times U_2(i) \, .$$
 and
\[
\bar{\pi}_{i}:=\sup_{ (u_1,u_2)\in U_1(i)\times U_2(i)} [-\bar{\pi}_{ii}(u_1,u_2)]<\infty  \, .
\]
	\item Finally, the measurable function $\bar{c}_k:K \to \mathbb{R}_{+}$ denotes the cost rate function for player $k, \; k=1,2$.
\end{itemize}

We consider a  continuous time Markov decision processes (CTMDPs)  $\{Y(t)\}_{t\geq 0}$ with state space $S$ and controlled rate matrix $\Pi_{u_1,u_2}=(\bar{\pi}_{ij}(u_1,u_2))$.
 To construct the underlying CTMDPs $Y(t)$ (as in [\cite{GP}, \cite{K1}, \cite{PZ}) we introduce some notations: let $S_\Delta:=S \cup \{\Delta\}$ (with some $\Delta \notin S$), $\Omega_0:=(S\times(0,\infty))^\infty$, $\Omega_m:=(S\times (0,\infty))^m\times S\times (\{\infty\}\times\{\Delta\})^\infty$ for $m\geq 1$ and $\Omega:=\cup_{m=0}^\infty\Omega_m$. Let $\mathscr{F}$ be the Borel $\sigma$-algebra on $\Omega$. Then we obtain the measurable space $(\Omega, \mathscr{F})$.
	For some $m\geq 1$, and sample $ \omega:=(i_0, \theta_1, i_1, \cdots , \theta_m, i_m, \cdots)\in \Omega,$ define
	 \begin{align*}
	 T_0(\omega):=0,~ T_n(\omega):= T_{n-1}(\omega)+\theta_{n},~ T_\infty(\omega):=\lim_{n\rightarrow\infty}T_n(\omega).
	\end{align*}
 Using $\{T_m\}$, we define the state process $\{Y(t)\}_{t\geq 0}$ as
	\begin{equation}
Y(t):=\sum_{m\geq 0}I_{\{T_m\leq t<T_{m+1}\}}i_m+ I_{\{t\geq T_\infty\}}\Delta, \text{ for } t\geq 0~(\text{with}~ T_0:=0).\label{eq 2.4}
	\end{equation}
	Here, $I_{E}$ denotes the indicator function of a set $E$, and we use the convention that $0+z=:z$ and $0z=:0$ for all $z\in S_\Delta$. Obviously, $Y(t)$ is right-continuous on $[0,\infty)$. From (\ref{eq 2.4}), we see that $T_m(\omega)$ $(m\geq 1)$ denotes the $m$-th jump moment of $\{Y(t)\}_{t\geq 0}$ and $i_{m-1}$ is the state of the process on $[T_{m-1}(\omega),T_m(\omega))$, $\theta_m(\omega)=T_m(\omega)-T_{m-1}(\omega)$ plays the role of sojourn time at state $i_{m-1}$, and the sample path $\{Y(t)(\omega)\}_{t\geq 0}$ has at most denumerable states $i_m(m=0,1,\cdots)$.  The process after $T_\infty$ is regarded to be absorbed in the state $\Delta$. Thus, let $q(\cdot | \Delta, u_1^\Delta,u_2^\Delta):\equiv 0$, $U_1^\Delta:=U_1\cup \{u_1^\Delta\}$, $U_2^\Delta:=U_2\cup \{u_2^\Delta\}$, $ U_1(\Delta):=\{u_1^\Delta\}$, $U_2(\Delta):=\{u_2^\Delta\}$. Also, assume that $\bar{c}_k(\Delta, u_1,u_2):\equiv 0$ ($\bar{c}_k$ is the running cost function for kth player) for all $(u_1,u_2)\in U_1^\Delta\times U_2^\Delta$, where $u_1^\Delta$, $u_2^\Delta$ are isolated points. Moreover, let $\mathscr{F}_t:=\sigma(\{T_m\leq s,Y({T_m})\in S\}:0\leq s\leq t, m\geq0)$ for all $t\geq 0$, $\mathscr{F}_{s-}=:\bigvee_{ t<s}\mathscr{F}_t$, and $\tilde{\mathscr{F}}:=\sigma(A\times \{0\},B\times (s,\infty):A\in \mathscr{F}_0, B\in \mathscr{F}_{s-})$ which denotes the $\sigma$-algebra of predictable sets on $\Omega\times [0,\infty)$ related to $\{\mathscr{F}_t\}_{t\geq 0}$.\\
To complete the specification of a risk-sensitive stochastic game  problem, we need, of course, to introduce an optimality criterion. This requires to define the class of strategies as below.
\begin{definition}
	A admissible strategy for player 1, denoted by $v_1=\{v_1(t)\}_{t\geq 0}$, is a transition probability $v_1(d u_1| \omega, t)$ from $(\Omega\times[0,\infty),\tilde{\mathscr{F}})$ onto $(U_1^\Delta,\mathcal{B}(U_1^\Delta))$, such that $v_1(U_1(Y(t-)(\omega))| \omega, t) = 1$. The set of all admissible strategies for player 1 is denoted by $\mathcal{A}_1$.
	A strategy $v_1 \in \mathcal{A}_1$, is called a Markov for player 1 if $v_1(t)( \omega)=v_1(t, Y(t-)(w))$, i.e., $v_1(d u_1|\omega, t) = v_1(d u_1| Y(t-)(w), t)$ for every $w\in \Omega$ and $t\geq 0$, where $Y({t-})(w):=\lim_{s\uparrow t} Y(s)(w)$. We denote by  $\mathcal{M}_1$ the family of all Markov strategies for player 1.
	If the Markov strategy $v_1$  for player 1 does not have any
explicit time dependency then it is called a stationary Markov strategy. The set of such strategies for player 1 is denoted by $\mathcal{S}_1$.  The sets of all admissible strategies $\mathcal{A}_2$, all Markov strategies $\mathcal{M}_2$ and all stationary strategies $\mathcal{S}_2$ for player 2 are defined similarly.
\end{definition}
To avoid the explosion of the state process $\{Y(t)\}_{t\geq 0}$, we need the following assumption imposed on the transition rates, which had been widely used in CTMDPs; see, for instance, [\cite{GL}, \cite{GLZ}, \cite{GP}, \cite{GS1}] and references therein.
\begin{assumption}\label{assm 3.1}
	There exists a Lyapunov function $\tilde{W} : S \to [1,\infty)$ such that
	\begin{enumerate}
		\item [(i)] $\sum_{j\in S}\tilde{W}(j)\overline{\pi}_{ij}(u_1,u_2)\leq C_1 \tilde{W}(i)+C_2$ for all $( u_1,u_2)\in U_1(i)\times U_2(i)$ and $i\in S$ with some constants $C_1\neq 0$, $C_2\geq 0$;
		
		\item [(ii)] $\bar{\pi}_{i}\leq C_3 \tilde{W}(i)$ for all $i\in S$ with some positive constant $C_3$.
	\end{enumerate}
\end{assumption}
For the rest of this article we are going to assume that Assumption \ref{assm 3.1} holds. Note that if $\sup_{i \in S}\bar{\pi}_{i}<\infty$ then Assumption \ref{assm 3.1} holds. In this case we can choose $\tilde{W}$ to be a suitable constant. Also note that under Assumption \ref{assm 3.1},  for any initial state $i\in S$ and any pair of strategies $(v_1,v_2)\in \mathcal{A}_1\times \mathcal{A}_2$, Theorem 4.27 in \cite{KR} yields the existence of a unique probability measure denoted by $P^{v_1,v_2}_i$ on $(\Omega,\mathscr{F})$.  Let $E^{v_1,v_2}_i$ be the expectation operator with respect to  $P^{v_1,v_2}_i$. Also, from [\cite{GH4}, pp.13-15], we know that $\{Y(t)\}_{t\geq 0}$ is a Markov process under any $(v_1,v_2)\in \mathcal{M}_1\times \mathcal{M}_2$ (in fact, strong Markov).

For any compact metric space $A$, let $\mathcal{P}(A)$ denote the space of probability measures on $A$ with Prohorov topology. Let $V_k = \mathcal{P}(U_k)$ and $V_k(i)=\mathcal{P}(U_k(i)) $ for $i\in S$ and $k=1,2$. For each $i ,j\in  S$, $k=1,2,$ $v_1\in V_1(i)$ and $v_2 \in V_2(i)$, the associated transition and cost rates are defined, respectively, as follows:
$$\pi_{ij}(v_1,v_2):=\int_{U_1(i)}\int_{U_2(i)}\bar{\pi}_{ij}(u_1,u_2)v_1(du_1)v_2(du_2),$$
$$c_{k}(v_1,v_2):=\int_{U_1(i)}\int_{U_2(i)}\bar{c}_{k}(u_1,u_2)v_1(du_1)v_2(du_2).$$

Note that for $k= 1, 2, v_k \in {\mathcal{S}_k}$ can be identified with a map $v_k: S \to V_k$ such that for each $j \in S, v_k(j) \in V_k(j)$
for each $j \in S$.
The sets ${\mathcal{S}_1}$ and ${\mathcal{S}_2}$ are endowed with product topology.

We list the commonly used notations below.
\begin{itemize}
	\item For any finite set $\mathcal{D}\subset S$, we define $\mathcal{B}_{\mathcal{D}} = \{f:S\to\mathbb{R}\mid f\,\,\,\text{is borel measurable function and}\,\,\, f(i) = 0\,\,\, \forall \,\, i\in \mathcal{D}^c\}$\,.
	\item  Given any real-valued function $\mathcal{V}\geq 1$ on $S$, we define a Banach space $(L^\infty_{\mathcal{V}},\|\cdot\|^\infty_\mathcal{V})$ of $\mathcal{V}$-weighted  functions by
	$$L^\infty_\mathcal{V}=\biggl\{u: S\rightarrow\mathbb{R}\mid \|u\|^\infty_\mathcal{V}:=\sup_{i\in  S}\frac{|u(i)|}{\mathcal{V}(i)}< \infty\biggr\}.$$
	\item $L^{1,\infty}_\mathcal{V}$ denotes the subset of $L^\infty_\mathcal{V}$ consists of function $u$ such that $ \|u\|^\infty_\mathcal{V}\leq 1$.
\end{itemize}

For $k=1,2$, let $\bar{c}_k: S \times U_1 \times U_2 \rightarrow [0, \ \infty)$ be the running cost function for the $k$th player, i.e., when state of the system is $i$ and the actions $(u_1,u_2)$ are chosen by the players, then  the cost incurred by the $k$th player is $\bar{c}_k(i,u_1,u_2)$. By choosing appropriate strategies, each player wants to minimize his/her accumulated cost over infinite time horizon.

For a pair of admissible strategies  $(v_1,v_2)$, the risk-sensitive ergodic cost for player $k$ is given by
\begin{equation}\label{main1}
\rho^{v_1,v_2}_k (i)\ := \   \limsup_{ T \to \infty} \frac{1}{T} \ln E_{i}^{v_1,v_2} \Big[ e^{\int^T_0  c_k(Y(t),v_1 (t),v_2 (t)) dt}
 \Big] \, ,
\end{equation}
where $Y(t)$ is the CTMDP corresponding to $(v_1,v_2) \in  \mathcal{A}_1 \times \mathcal{A}_2 $ and  $E_{i}^{v_1,v_2} $ denotes the expectation with respect to the law of the process $Y(t)$ with initial condition $Y(0)=i$.

Since we are allowing our transition and cost rates to be unbounded, to guarantee the finiteness of $\rho^{v_1,v_2}_k$ for $k=1,2$, we need the following Assumption.
\begin{assumption}\label{assm 3.2}
We assume that the CTMDP $\{Y(t)\}_{t\geq 0}$ is irreducible under every  pair of stationary Markov strategies $(v_1,v_2)\in \mathcal{S}_1 \times \mathcal{S}_2$.
Furthermore, suppose there exist a constant $C_4>0$ and a Lyapunov function $W : S \to [1,\infty)$ such that one of the following hold.
\begin{itemize}
\item[(a)]\textbf{When the running cost is bounded:} For some positive constant $\gamma > \max\{\|c_1\|_{\infty},\|c_2\|_{\infty}\}$  and a finite set $\mathscr{K}$ it holds that
$$\sup_{(u_1,u_2)\in U_1(i)\times U_2(i)}\sum_{j\in S}W(j)\overline{\pi}_{ij}(u_1,u_2)\leq C_4 I_{\mathscr{K}}(i)-\gamma W(i) ~\forall i\in S.$$
\item[(b)]\textbf{When the running cost is unbounded:} For some norm-like function $\ell  :S\rightarrow\mathbb{R}_{+}$ and a finite set $\mathscr{K}$ it holds that
$$\sup_{(u_1,u_2)\in U_1(i)\times U_2(i)}\sum_{j\in S}W(j)\overline{\pi}_{ij}(u_1,u_2)\leq C_4 I_{\mathscr{K}}(i)-\ell(i)W(i) ~\forall i\in S.$$
Also, the functions $\ell(\cdot)-\max_{(u_1,u_2)\in U_1(\cdot)\times U_2(\cdot)}c_k(\cdot,u_1,u_2), \; k=1,2$, are norm-like.
\end{itemize}
\end{assumption}
\begin{definition}
A pair of strategies $(v_1^{*},v_2^{*}) \in \mathcal{A}_1 \times \mathcal{A}_2$ is called a Nash equilibrium if
\begin{equation*}
\rho_{1}^{v_1^*,v^*_2} (i)\ \leq \  \rho_{1}^{v_1,v^*_2} (i)\;\mbox{for all} \; v_1\in \mathcal{A}_1\; \mbox{ and} \; i \in S
\end{equation*}
and
\begin{equation*}
\rho_{2}^{v_1^*,v^*_2} (i)\ \leq \  \rho_{2}^{v^*_1,v_2} (i)\;\mbox{for all} \; v_2\in \mathcal{A}_2\; \mbox{ and} \; i \in S.
\end{equation*}
\end{definition}
\noindent We wish to establish the  existence of a Nash equilibrium in stationary strategies.
To ensure the existence of a Nash equilibrium, we assume the following:
\begin{assumption}\label{assm 3.3}
	\begin{enumerate}
		\item [(i)] For any fixed $i,  j\in S,\, $k=1,2$\,,  ~ \overline{\pi}_{ij}(u_1,u_2)$ and $\bar{c}_k(i, u_1,u_2)$ are continuous in $(u_1,u_2)\in U_1(i)\times U_2(i)$\,.
		
		\item [(ii)]$\displaystyle \sum_{j\in S}W(j)\overline{\pi}_{ij}(u_1,u_2)$ is continuous in $(u_1,u_2)\in U_1(i)\times U_2(i)$ for any given $i\in S$, where $W$ is as Assumption \ref{assm 3.2}.
		\item [(iii)] There exists $i_0\in S$ such that $\overline{\pi}_{i_0j}(u_1,u_2)>0$ for all $j\neq i_0$ and $(u_1,u_2)\in U_1(j)\times U_2(j)$.
	\end{enumerate}	
\end{assumption}

We now proceed to establish the existence of a Nash equilibrium in stationary strategies.
To this end we first outline a procedure for establishing the existence of a Nash equilibrium. Suppose player 2 announces that he is going to employ a strategy $v_2 \in \mathcal{S}_2$. In such a scenario, player 1 attempts to minimize
\begin{equation*}
\rho^{v_1,v_2}_1 (i)\ = \   \limsup_{ T \to \infty} \frac{1}{ T} \ln E_{i}^{v_1,v_2} \Big[ e^{ \int^T_0  c_1(Y(t),v_1 (t),v_2 (Y(t-))) dt}
 \Big] \, ,
\end{equation*}
over $v_1 \in \mathcal{A}_1$. Thus for player 1 it is a continuous time Markov decision problem (CTMDP) with risk sensitive ergodic cost. This problem has been studied in  \cite{BP}, \cite{GS}, \cite{KP1}, \cite{KP2}. In particular under certain assumptions, it is shown in \cite{BP}, \cite{KP1}, \cite{KP2},  that the following Hamilton-Jacobi-Bellman (HJB) equation
 \begin{equation*}
 \left\{\begin{aligned}
  \rho_{1} ~\hat{\psi}_{1}(i) &= \inf_{v_1\in V_1(i)} \Big [ \sum_{j\in S}{\pi}_{ij}({v_1,v_{2}(i)})
\hat\psi_{1}(j) + c_1(i,v_1,v_{2}(i))\hat\psi_{1}(i) \Big ]   \\
 \hat\psi_{1}(i_0)  &= 1,
 \end{aligned}
 \right.
\end{equation*}
has a suitable solution $(\rho_{1} ,\hat{\psi}_{1})$, where $\rho_{1} $ is a scalar and $\hat{\psi}_{1}:S \to \mathbb{R}$ has suitable growth rate; $i_0$ is a fixed element of $S$.  Furthermore it is shown in \cite{BP}, \cite{KP1}, \cite{KP2} that
\begin{equation*}
\rho_1 \ = \  \inf_{v_1 \in \mathcal{A}_1} \limsup_{ T \to \infty} \frac{1}{ T} \ln E_{i}^{v_1,v_2} \Big[ e^{ \int^T_0  c_1(Y(t),v_1 (t),v_2 (Y(t-))) dt}
 \Big] \, ,
\end{equation*}
and if $v_1^* \in \mathcal{S}_1$ is such that for $i \in S$
\begin{eqnarray*}
 && \inf_{v_1\in V_1(i)} \Big [ \sum_{j\in S}{\pi}_{ij}({v_1,v_{2}(i)})
\hat\psi_{1}(j) + c_1(i,v_1,v_{2}(i))\hat\psi_{1}(i) \Big ]  \nonumber \\
&=&  \sum_{j\in S}{\pi}_{ij}({v_{1}^*(i),v_{2}(i)})
\hat\psi_{1}(j) + c_1(i,v_{1}^*(i),v_{2}(i))\hat\psi_{1}(i) ,
\end{eqnarray*}
then $v_1^* \in \mathcal{S}_1$ is an optimal control for player 1, i.e., for any $i \in S $
\begin{equation*}
\rho_1 \ = \   \limsup_{ T \to \infty} \frac{1}{ T} \ln E_{i}^{v_1^*,v_2} \Big[ e^{\int^T_0  c_1(Y(t),v_1^* (Y(t-)),v_2 (Y(t-))) dt}
 \Big] \, .
\end{equation*}
 In other words, given that player 2 is using the strategy $v_2 \in \mathcal{S}_2$, $v_1^* \in \mathcal{S}_1$ is an optimal response for player 1. Clearly $v_1^*$ depends on $v_2$ and moreover there may be several optimal responses for player 1 in $\mathcal{S}_1$. Analogous results holds for player 2 if player 1 announces that he is going to use a strategy $v_1 \in \mathcal{S}_1$. Hence given a pair of strategies $(v_1,v_2) \in \mathcal{S}_1 \times \mathcal{S}_2$, we can find a set of pairs of optimal responses $\{(v_1^*,v_2^*) \in \mathcal{S}_1 \times \mathcal{S}_2\}$ via the appropriate pair of HJB equations described above. This defines a set-valued map. Clearly any fixed point of this set-valued map is a Nash equilibrium.

 The above discussion leads to the following procedure for finding a pair of Nash equilibrium strategies.
 Suppose that there exist a pair of stationary strategies  $(v_1^*,v_2^*) \in \mathcal{S}_1 \times \mathcal{S}_2$, a pair of scalars $(\rho_1^*, \rho_2^*)$ and a pair of functions $(\hat{\psi}_{1}^*,\hat{\psi}_{2}^*)$ with appropriate growth conditions, satisfying the following coupled HJB equations:
 \begin{equation*}
 \left\{\begin{aligned}
         \rho^{*}_{1} ~\hat{\psi}^{*}_{1}(i) &= \inf_{v_1\in V_1(i)} \Big [ \sum_{j\in S}{\pi}_{ij}({v_1,v_{2}^*(i)})
\hat\psi^{*}_{1}(j) + c_1(i,v_1,v_{2}^*(i))\hat\psi^{*}_{1}(i) \Big ]  \\
       &=  \sum_{j\in S}{\pi}_{ij}({v_{1}^*(i),v_{2}^*(i)})
\hat\psi^{*}_{1}(j) + c_1(i,v_{1}^*(i),v_{2}^*(i))\hat\psi^{*}_{1}(i)    \\
\displaystyle{ \hat\psi^{*}_{1}(i_0) } &= 1,  \\
 \rho^{*}_{2} ~\hat{\psi}^{*}_{2}(i) &= \inf_{v_2\in V_2(i)} \Big [ \sum_{j\in S}{\pi}_{ij}({v_{1}^*(i),v_2})\hat\psi^{*}_{2}(j) + c_2(i,v_{1}^*(i),v_{2})\hat\psi^{*}_{2}(i) \Big ] \\
&=  \sum_{j\in S}{\pi}_{ij}({v_{1}^*(i),v_{2}^*(i)})
\hat\psi^{*}_{2}(j) + c_2(i,v_{1}^*(i),v_{2}^*(i))\hat\psi^{*}_{2}(i)   \\
\displaystyle{ \hat\psi^{*}_{2}(i_0) } &= 1,
       \end{aligned}
 \right.
\end{equation*}
where as before $i_0 \in S$ is a fixed point. Then it can be shown that
$(v_1^*,v_2^*)$ is a pair of Nash equilibrium and  $(\rho_1^*, \rho_2^*)$ is the pair of corresponding Nash values. Thus the main result of our paper is to establish that the above coupled HJB equations has suitable solutions.
\begin{remark}
 Note that the similar stochastic optimal control problem has been studied in \cite{GS}, \cite{KP2}  for bounded cost and bounded transition rates. But in our game model transition and cost rates are unbounded. Analogous MDP problems are treated in \cite{BP}.
\end{remark}

\section{Coupled HJB Equations and Existence of Nash Equilibrium }
By the definition of weak convergence of probability measures, one can easily get the following result, which will be crucial for the existence of Nash equilibrium; we omit the details.
\begin{lemma}\label{lemm 3.1}
	Under Assumptions \ref{assm 3.1}, \ref{assm 3.2}, and \ref{assm 3.3}, the functions
	$$c_k(i,v_1,v_2), \; k=1,2 \; \; \mbox{and} \; \; \sum_{j\in S}\pi_{ij}(v_1,v_2)\phi(j)$$ are continuous  on $V_1(i)\times V_2(i)$ for each fixed $\phi \in L^\infty_{W}$ and $i\in S$.
\end{lemma}

Let $\mathcal{D}_n\subset S$ be an increasing sequence of finite sets such that $\cup_{n}\mathcal{D}_n = S$ and $i_0\in \mathcal{D}_n$ for each $n\geq 1$\,. In the next lemma we show the existence of eigenpairs to certain equations in $\mathcal{D}_n$ for each $n\in \mathbb{N}$\,.
\begin{lemma}\label{L3.2B}
Grant Assumptions \ref{assm 3.1}, \ref{assm 3.2}, and \ref{assm 3.3}. Then for each $n\in\mathbb{N}$, the following hold.
\begin{enumerate}
\item For $\hat{v}_2\in \mathcal{S}_2$, there exists an eigenpair $(\rho_{1,n},\psi_{1,n})\in \mathbb{R}\times \mathcal{B}_{\mathcal{D}_n}^+$, satisfying		
		\begin{equation}\label{EL3.2BA}
 \left\{\begin{aligned}
\rho_{1,n}\psi_{1,n}(i)&=\inf_{v_1\in V_1(i)}\bigg[\sum_{j\in S}\psi_{1,n}(j){\pi}_{ij}(v_1,\hat{v}_2(i))+c_1(i,v_1,\hat{v}_2(i))\psi_{1,n}(i)\bigg]~\text{for}~i\in \mathcal{D}_n,\\
\psi_{1,n}(i_0) &= 1.
 \end{aligned}
 \right.
\end{equation}
Moreover, we have		
\begin{equation}\label{EL3.2BB}
0 \leq \liminf_{n\to\infty}	\rho_{1,n} \leq \limsup_{n\to\infty}\rho_{1,n} \leq   \inf_{v_1 \in \mathscr{A}_1} \limsup_{ T \to \infty} \frac{1}{ T} \ln E_{i_0}^{v_1,\hat{v}_2} \Big[ e^{ \int^T_0  c_1(Y(t),v_1 (t),\hat{v}_2 (Y(t-))) dt}\Big].
\end{equation}			
	\item 	Similarly, for $\hat{v}_1\in \mathcal{S}_1$, there exists an eigenpair $(\rho_{2,n},\psi_{2,n})\in \mathbb{R}\times \mathcal{B}_{\mathcal{D}_n}^+$, satisfying	\begin{equation}\label{EL3.2BC}
 \left\{\begin{aligned}
\rho_{2,n}\psi_{2,n}(i)&=\inf_{v_2\in V_2(i)}\bigg[\sum_{j\in S}\psi_{2,n}(j)\pi_{ij}(\hat{v}_1(i),v_2)+c_2(i,\hat{v}_1(i),v_2)\psi_{2,n}(i)\bigg]~\text{for}~i\in \mathcal{D}_n,\\
\psi_{2,n}(i_0) &= 1.
 \end{aligned}
 \right.
\end{equation}
Moreover, we have
\begin{equation}\label{EL3.2BD}
0\leq \liminf_{n\to\infty}\rho_{2,n}\leq \limsup_{n\to\infty}\rho_{2,n} \leq  \inf_{v_2 \in\mathscr{A}_2} \limsup_{ T \to \infty} \frac{1}{ T} \ln E_{i_0}^{\hat{v}_1,v_2} \Big[ e^{ \int^T_0  c_2(Y(t),\hat{v}_1 (Y(t-)),v_2 (t)) dt}\Big].
\end{equation}
\end{enumerate}
\end{lemma}
\begin{proof}
Follows by analogous arguments as in \cite[Lemma~3.1, Lemma~3.3]{BP}. We omit the details.
\end{proof}
Next by taking limit $n\to \infty$ in the equations we show that the limiting equations admit eigenpairs in appropriate spaces. In particular, we have the following theorem.

\begin{thm}\label{theo 3.1}
Grant Assumptions \ref{assm 3.1}, \ref{assm 3.2}, and \ref{assm 3.3}. Then the following hold.
	\begin{enumerate}
		\item For $\hat{v}_2\in \mathcal{S}_2$, there exists a unique minimal eigenpair $(\rho_{1},\psi_{1})\in \mathbb{R}_{+}\times L^{1,\infty}_{W}$, $\psi_{1}>0$, satisfying		
		\begin{equation}\label{eq 3.1}
 \left\{\begin{aligned}
\rho_{1}\psi_{1}(i)&=\inf_{v_1\in V_1(i)}\bigg[\sum_{j\in S}\psi_{1}(j){\pi}_{ij}(v_1,\hat{v}_2(i))+c_1(i,v_1,\hat{v}_2(i))\psi_{1}(i)\bigg]~\text{for}~i\in S,\\
\psi_{1}(i_0) &= 1.
 \end{aligned}
 \right.
\end{equation}
Moreover, we have		
 \begin{equation}
			\rho_1 \ = \  \inf_{v_1 \in \mathscr{A}_1} \limsup_{ T \to \infty} \frac{1}{ T} \ln E_{i}^{v_1,\hat{v}_2} \Big[ e^{ \int^T_0  c_1(Y(t),v_1 (t),\hat{v}_2 (Y(t-))) dt}
			\Big] (:=\rho_{1}^{\hat{v}_2}=\inf_{v_1 \in \mathscr{A}_1}\rho_{1}^{v_1,\hat{v}_2}),\label{eq 3.3}
			\end{equation}	
			
	and there exists a finite set $\mathscr{B}_1\supset \mathscr{K}$, such that
			\begin{align}
			\psi_{1}(i)=\inf_{v_1\in\mathcal{S}_1}E^{v_1,\hat{v}_2}_i\bigg[e^{\int_{0}^{\hat{\tau}(\mathscr{B}_1)}(c_1(Y(t),v_1(Y(t-)),\hat{v}_2(Y(t-)))-\rho_{1})dt}\psi_{1}(Y({\hat{\tau}(\mathscr{B}_1)}))\bigg](:=\psi_{1}^{\hat{v}_2}(i))~\forall i\in \mathscr{B}_1^c,\label{eq 3.2}
			\end{align}
			where $\hat{\tau}(\mathscr{B}_1)= \tau(\mathscr{B}_1^c)= \inf\{t:Y(t)\in \mathscr{B}_1\}=:\tilde{\tau}_1$.

	\item 	Similarly, for $\hat{v}_1\in \mathcal{S}_1$, there exists a unique minimal eigenpair $(\rho_{2},\psi_{2})\in \mathbb{R}_{+}\times L^{1,\infty}_{W}$, $\psi_{2}>0$ satisfying	
	\begin{equation}\label{eq 3.4}
 \left\{\begin{aligned}
\rho_{2}\psi_{2}(i)&=\inf_{v_2\in V_2(i)}\bigg[\sum_{j\in S}\psi_{2}(j)\pi_{ij}(\hat{v}_1(i),v_2)+c_2(i,\hat{v}_1(i),v_2)\psi_{2}(i)\bigg]~\text{for}~i\in S,\\
\psi_{2}(i_0) &= 1.
 \end{aligned}
 \right.
\end{equation}
		Moreover, we have
		\begin{equation}
		\rho_2\ = \  \inf_{v_2 \in\mathscr{A}_2} \limsup_{ T \to \infty} \frac{1}{ T} \ln E_{i}^{\hat{v}_1,v_2} \Big[ e^{ \int^T_0  c_2(Y(t),\hat{v}_1 (Y(t-)),v_2 (t)) dt}
		\Big] (:=\rho_{2}^{\hat{v}_1}=\inf_{v_2 \in \mathscr{A}_2}\rho_{2}^{\hat{v}_1,v_2}) ,\label{eq 3.6}
		\end{equation}
		and there exists a finite set $\mathscr{B}_2\supset \mathscr{K}$,  such that
		\begin{align}
		\psi_{2}(i)=\inf_{v_2\in\mathcal{S}_2}E^{\hat{v}_1,v_2}_i\bigg[e^{\int_{0}^{\hat{\tau}(\mathscr{B}_2)}(c_2(Y(t),\hat{v}_1(Y(t-)),v_2(Y(t-)))-\rho_{2})dt}\psi_{2}(Y({\hat{\tau}(\mathscr{B}_2)}))\bigg](:=\psi_{2}^{\hat{v}_1}(i))~\forall i\in \mathscr{B}_2^c,\label{eq 3.5}
		\end{align}
		where $\hat{\tau}(\mathscr{B}_2)=\tau(\mathscr{B}_2^c)=\inf\{t:Y(t)\in \mathscr{B}_2\}=:\tilde{\tau}_2$.
	\end{enumerate}
\end{thm}
\begin{proof}
Since $c_1 \geq 0$, using Assumption \ref{assm 3.2}, we deduce that there exists a finite set $\mathscr{B}_1$ containig $\mathscr{K}$ such that
\begin{itemize}
\item Under Assumption  \ref{assm 3.2} (a)
\begin{equation*}
(\sup_{(u_1,u_2)\in U_{1}(i)\times U_{2}(i)} c_{1}(i, u_1, u_2) - \rho_{1,n}) < \gamma \quad \forall \,\,\, i\in \mathscr{B}_1^{c} \quad\text{and all $n$ large enough}\,.
\end{equation*}
\item Under Assumption  \ref{assm 3.2} (b)
\begin{equation*}
(\sup_{(u_1,u_2)\in U_{1}(i)\times U_{2}(i)} c_{1}(i, u_1, u_2) - \rho_{1,n}) < \ell(i) \quad \forall \,\,\, i\in \mathscr{B}_1^{c} \quad\text{and all $n$ large enough}\,.
\end{equation*}
\end{itemize}
Then applying It\^{o}-Dynkin formula, from Assumption \ref{assm 3.2}, we have the following estimates:
\begin{itemize}
\item Under Assumption \ref{assm 3.2}(a):
\begin{equation}\label{ELY1A}
E^{v_1, \hat{v}_2}_i\bigg[e^{\hat{\tau}(\mathscr{B}_1)\gamma}W(Y({\hat{\tau}(\mathscr{B}_1)}))\bigg]\leq W(i)\,\,~\forall i\in \mathscr{B}_1^c\,.
\end{equation}
\item Under Assumption \ref{assm 3.2}(b):
\begin{equation}\label{ELY1B}
E^{v_1, \hat{v}_2}_i\bigg[e^{\int_{0}^{\hat{\tau}(\mathscr{B}_1)}\ell(Y(t))dt}W(Y({\hat{\tau}(\mathscr{B}_1)}))\bigg] \leq W(i)\,\,~\forall i\in \mathscr{B}_1^c\,.
\end{equation}
\end{itemize}
Now as in \cite[Lemma~3.4]{BP}, using the Lyapunov function $W$ we construct a barrier. Then following arguments similar to \cite[Lemma~3.4]{BP} and letting $n\to\infty$, there exists $(\rho_{1},\psi_{1})\in \mathbb{R}_{+}\times L^{1,\infty}_{W}$, $\psi_{1}>0$, satisfying (\ref{eq 3.1}). By truncating the running cost $c_1$, one can show that $\rho_{1}$ satisfies (\ref{eq 3.3}) (see, \cite[Lemma~3.5]{BP})\,.

Next we prove the stochastic representation (\ref{eq 3.2}). Applying It\^{o}-Dynkin formula and Fatou's lemma, for any minimizing selector $v_1^*$ of (\ref{eq 3.1}) we have
\begin{align}\label{ET3.1BA1}
\psi_{1}(i) &\geq E^{v_1^*,\hat{v}_2}_i\bigg[e^{\int_{0}^{\hat{\tau}(\mathscr{B}_1)}(c_1(Y(t),v_1^*(Y(t-)),\hat{v}_2(Y(t-)))-\rho_{1})dt}\psi_{1}(Y({\hat{\tau}(\mathscr{B}_1)}))\bigg]\nonumber\\
&\geq \inf_{v_1\in \mathcal{S}_1}E^{v_1,\hat{v}_2}_i\bigg[e^{\int_{0}^{\hat{\tau}(\mathscr{B}_1)}(c_1(Y(t),v_1(Y(t-)),\hat{v}_2(Y(t-)))-\rho_{1})dt}\psi_{1}(Y({\hat{\tau}(\mathscr{B}_1)}))\bigg]~\forall i\in \mathscr{B}_1^c\,.
\end{align}
Again, by applying It\^{o}-Dynkin formula, from (\ref{EL3.2BA}) for any $v_1\in \mathcal{S}_1$, $T>0$ and $i\in \mathcal{D}_n\cap \mathscr{B}_1^c$ it follows that
\begin{align}\label{ET3.1BA2}
\psi_{1,n}(i) &\leq E^{v_1,\hat{v}_2}_i\bigg[e^{\int_{0}^{\hat{\tau}(\mathscr{B}_1)\wedge\tau(\mathcal{D}_n)\wedge T}(c_1(Y(t),v_1(Y(t-)),\hat{v}_2(Y(t-)))-\rho_{1,n})dt}\psi_{1,n}(Y(\hat{\tau}(\mathscr{B}_1)\wedge\tau(\mathcal{D}_n)\wedge T))\bigg]\nonumber\\
&\leq E^{v_1,\hat{v}_2}_i\bigg[e^{\int_{0}^{\hat{\tau}(\mathscr{B}_1)}(c_1(Y(t),v_1(Y(t-)),\hat{v}_2(Y(t-)))-\rho_{1,n})dt}\psi_{1,n}(Y({\hat{\tau}(\mathscr{B}_1)}))I_{\{\hat{\tau}(\mathscr{B}_1)\leq \tau(\mathcal{D}_n)\wedge T\}}\bigg]\nonumber\\
&\,\,\, + E^{v_1,\hat{v}_2}_i\bigg[e^{\int_{0}^{T}(c_1(Y(t),v_1(Y(t-)),\hat{v}_2(Y(t-)))-\rho_{1,n})dt}\psi_{1,n}(Y(T))I_{\{T \leq\hat{\tau}(\mathscr{B}_1)\wedge \tau(\mathcal{D}_n)\}}\bigg]\,.
\end{align}
Using (\ref{ELY1A}) and the fact that $\psi_{1,n} \leq W$ (by our construction), we have
\begin{align*}
&E^{v_1,\hat{v}_2}_i\bigg[e^{\int_{0}^{T}(c_1(Y(t),v_1(Y(t-)),\hat{v}_2(Y(t-)))-\rho_{1,n})dt}\psi_{1,n}(Y(T))I_{\{T \leq\hat{\tau}(\mathscr{B}_1)\wedge \tau(\mathcal{D}_n)\}}\bigg] \nonumber\\
&\leq e^{(\|c_{1}\|_{\infty} - \rho_{1,n} - \gamma)T}E^{v_1,\hat{v}_2}_i\bigg[e^{T\gamma}W(Y(T))I_{\{T \leq\hat{\tau}(\mathscr{B}_1)\wedge \tau(\mathcal{D}_n)\}}\bigg]\nonumber\\
&\leq e^{(\|c_{1}\|_{\infty} - \rho_{1,n} - \gamma)T} W(i)\,.
\end{align*}Thus, letting $T\to\infty$ from (\ref{ET3.1BA2}) we get
\begin{align*}
\psi_{1,n}(i)\leq E^{v_1,\hat{v}_2}_i\bigg[e^{\int_{0}^{\hat{\tau}(\mathscr{B}_1)}(c_1(Y(t),v_1((Y(t-))),\hat{v}_2(Y(t-)))-\rho_{1,n})dt}\psi_{1,n}(Y(\hat{\tau}(\mathscr{B}_1)))I_{\{\hat{\tau}(\mathscr{B}_1)\leq \tau(\mathcal{D}_n)\}}\bigg]
\end{align*}
Now, since $\psi_{1,n} \leq W$ using (\ref{ELY1A}) by dominated convergence theorem it follows that
\begin{align}\label{ET3.1BA3}
\psi_{1}(i)\leq E^{v_1,\hat{v}_2}_i\bigg[e^{\int_{0}^{\hat{\tau}(\mathscr{B}_1)}(c_1(Y(t),v_1(Y(t-)),\hat{v}_2(Y(t-)))-\rho_{1})dt}\psi_{1}(Y(\hat{\tau}(\mathscr{B}_1)))\bigg]~\forall i\in \mathscr{B}_1^c\,.
\end{align}Since $v_1\in\mathcal{S}_1$ is arbitrary, combining (\ref{ET3.1BA1}) and (\ref{ET3.1BA3}), we obtain (\ref{eq 3.2}). Also, it it clear from the proof that for any minimizing selector $v_1^*$ of (\ref{eq 3.1}) we have
\begin{align}\label{ET3.1B}
\psi_{1}(i)= E^{v_1^*,\hat{v}_2}_i\bigg[e^{\int_{0}^{\hat{\tau}(\mathscr{B}_1)}(c_1(Y(t),v_1^*(Y(t-)),\hat{v}_2(Y(t-)))-\rho_{1})dt}\psi_{1}(Y({\hat{\tau}(\mathscr{B}_1)}))\bigg]~\forall i\in \mathscr{B}_1^c\,.
\end{align} Using (\ref{ELY1B}) it is easy to check that the same conclusion holds under Assumption \ref{assm 3.2}(b)\,.

Now exploiting the stochastic representation (\ref{eq 3.2}), we show that $(\rho_{1},\psi_{1})\in \mathbb{R}_{+}\times L^{1,\infty}_{W}$ is the minimal eigenpair. Suppose $(\hat{\rho}_{1},\hat{\psi}_{1})\in \mathbb{R}_{+}\times L^{1,\infty}_{W}$, $\hat{\psi}_{1}>0$ is an eigenpair satisfying
\begin{equation}\label{ET3.1C}
 \left\{\begin{aligned}
\hat{\rho}_{1}\hat{\psi}_{1}(i)&=\inf_{v_1\in V_1(i)}\bigg[\sum_{j\in S}\hat{\psi}_{1}(j){\pi}_{ij}(v_1,\hat{v}_2(i))+c_1(i,v_1,\hat{v}_2(i))\hat{\psi}_{1}(i)\bigg]~\text{for}~i\in S,\\
\hat{\psi}_{1}(i_0) &= 1.
 \end{aligned}
 \right.
\end{equation}
 We want to show that $\rho_{1} \leq \hat{\rho}_{1}$. If not suppose that $\rho_{1} > \hat{\rho}_{1}$. Then, for any minimizing selector $\hat{v}_1^*$ of (\ref{ET3.1C}), applying It\^{o}-Dynkin formula and Fatou's lemma, we obtain
\begin{equation}\label{ET3.1D}
\hat{\psi}_{1}(i)\geq E^{\hat{v}_1^*,\hat{v}_2}_i\bigg[e^{\int_{0}^{\hat{\tau}(\mathscr{B}_1)}(c_1(Y(t),\hat{v}_1^*(Y(t-)),\hat{v}_2(Y(t-)))-\hat{\rho}_{1})dt}\hat{\psi}_{1}(Y({\hat{\tau}(\mathscr{B}_1)}))\bigg]~\forall i\in \mathscr{B}_1^c\,.
\end{equation} Whereas from (\ref{eq 3.2}), we have
\begin{equation}\label{ET3.1E}
\psi_{1}(i)\leq E^{\hat{v}_1^*,\hat{v}_2}_i\bigg[e^{\int_{0}^{\hat{\tau}(\mathscr{B}_1)}(c_1(Y(t),\hat{v}_1^*(Y(t-)),\hat{v}_2(Y(t-)))-\hat{\rho}_{1})dt}\psi_{1}(Y({\hat{\tau}(\mathscr{B}_1)}))\bigg]~\forall i\in \mathscr{B}_1^c\,.
\end{equation} Let $\hat{\kappa} := \min_{\mathscr{B}_1}\frac{\hat{\psi}_1}{\psi_1}$\,. Hence, from (\ref{ET3.1D}) and (\ref{ET3.1E}) it follows that $(\hat{\psi}_1 - \hat{\kappa}\psi_1) \geq 0$ in $S$ and $(\hat{\psi}_1 - \hat{\kappa}\psi_1)(\Tilde{i}_0) = 0$ for some $\Tilde{i}_0\in \mathscr{B}_1$\,. Now, combining (\ref{eq 3.1}) and (\ref{ET3.1C}) we deduce that
\begin{equation}\label{ET3.1F}
\bigg[\sum_{j\neq \Tilde{i}_0}(\hat{\psi}_1 - \hat{\kappa}\psi_1)(j){\pi}_{\Tilde{i}_0 j}(\hat{v}_1^*(\Tilde{i}_0),\hat{v}_2(\Tilde{i}_0))\bigg] \equiv 0\,.
\end{equation} Since $Y(t)$ is irreducible under $(\hat{v}_1^*,\hat{v}_2)$, in view of (\ref{ET3.1F}) it is clear that $(\hat{\psi}_1 - \hat{\kappa}\psi_1)\equiv 0$. Again, since $\hat{\psi}_1(i_0) = \psi_1(i_0) = $1, we get $\hat{\psi}_1 \equiv \psi_1$. But this is a contradiction to the fact that $\rho_{1} > \hat{\rho}_{1}$. Thus we deduce that $(\rho_{1},\psi_{1})\in \mathbb{R}_{+}\times L^{1,\infty}_{W}$ is the minimal eigenpair. Following the above argument one can show that any eigenfunction satisfying (\ref{eq 3.2}) is unique upto a scalar multiplication. Also, by the similar argument, one can show that there exists a minimal eigenpair $(\rho_{2},\psi_{2})\in \mathbb{R}_{+}\times L^{1,\infty}_{W}$ satisfying (\ref{eq 3.4}), (\ref{eq 3.6}) and (\ref{eq 3.5}). This completes the proof.
\end{proof}
To proceed further we establish some technical results needed later.

\begin{lemma}\label{lemm 3.4}
 Let Assumptions \ref{assm 3.1}, \ref{assm 3.2}, and \ref{assm 3.3} hold. Then the maps $\hat{v}_1\rightarrow\psi_{2}^{\hat{v}_1}$ from $\mathcal{S}_1\rightarrow L^\infty_{W}$, $\hat{v}_1\rightarrow\rho_{2}^{\hat{v}_1}$ from $\mathcal{S}_1\rightarrow \mathbb{R}_{+}$, $\hat{v}_2\rightarrow\psi_{1}^{\hat{v}_2}$ from $\mathcal{S}_2\rightarrow L^\infty_{W}$, and $\hat{v}_2\rightarrow\rho_{1}^{\hat{v}_2}$ from $\mathcal{S}_2\rightarrow \mathbb{R}_{+}$ are continuous.
\end{lemma}
\begin{proof}
	Let $\{v_{2,n}\}$ be a sequence in $\mathcal{S}_2$ such that $v_{2,n}\rightarrow \tilde{v}_2$ in $\mathcal{S}_2$, i.e., for each $i \in S, \, v_{2,n}(i) \to \tilde{v}_2(i) \, \, \mbox{in} \,  V_2(i) $.
	Now by Theorem \ref{theo 3.1},  there exists $(\rho_{1}^{v_{2,n}},\psi_{1}^{v_{2,n}})\in \mathbb{R}_{+}\times L^{1,\infty}_{W}$, $\psi_{1}^{v_{2,n}}>0$ satisfying
	\begin{align}
	\rho_{1}^{v_{2,n}}\psi_{1}^{v_{2,n}}(i)=\inf_{v_1\in V_1(i)}\biggl[\sum_{j\in S}\psi_{1}^{v_{2,n}}(j)\pi_{ij}(v_1,v_{2,n}(i))+c_1(i,v_1,v_{2,n}(i))\psi_{1}^{v_{2,n}}(i)\biggr],\label{eq 3.7}
	\end{align}
with $\psi_{1}^{v_{2,n}}(i_0)=1$.
	Now, since $\psi_{1}^{v_{2,n}}\in L^{1,\infty}_{W}$, by a standard diagonalization argument, there exists a function $\psi_{1}^{*}\in L^{1,\infty}_{W}$ such that $\psi_{1}^{v_{2,n}}(i)\rightarrow \psi_{1}^{*}(i)$ as $n\to\infty$ for all $i\in S$. Also, $\{	\rho_{1}^{v_{2,n}}\}$ is a bounded sequence. Hence, along a suitable subsequence (without loss of generality denoting by the same notation)  $\rho_{1}^{v_{2,n}}\rightarrow\rho_{1}^{*}$.
	Now from (\ref{eq 3.7}), for any $v_1\in V_1(i)$ we deduce that
		\begin{eqnarray*}
		\rho_{1}^{v_{2,n}}\psi_{1}^{v_{2,n}}(i)&\leq & \biggl[\sum_{j\in S}\psi_{1}^{v_{2,n}}(j)\pi_{ij}(v_1,v_{2,n}(i))+c_1(i,v_1,v_{2,n}(i))\psi_{1}^{v_{2,n}}(i)\biggr].
		\end{eqnarray*}
		This implies that
		\begin{align}\label{eq 3.8}
	\rho_{1}^{v_{2,n}}\psi_{1}^{v_{2,n}}(i)-\psi_{1}^{v_{2,n}}(i)\pi_{ii}(v_1,v_{2,n}(i))\leq  \biggl[\sum_{j\neq i}\psi_{1}^{v_{2,n}}(j)\pi_{ij}(v_1,v_{2,n}(i))+c_1(i,v_1,v_{2,n}(i))\psi_{1}^{v_{2,n}}(i)\biggr]. \nonumber \\
		\end{align}
	
	Note that	
			\begin{align}
	\sum_{j\neq i}\psi_{1}^{v_{2,n}}(j)\pi_{ij}(v_1,v_{2,n}(i))
	\leq \sum_{j\neq i}W(j)\pi_{ij}(v_1,v_{2,n}(i)).\label{eq 3.9}
	\end{align}
	Thus, using Lemma \ref{lemm 3.1}, generalized Fatou's lemma in \cite[Lemma~8.3.7]{HL} and taking $n\rightarrow\infty$ in (\ref{eq 3.8}), we get
	  \begin{align*}
	  \rho_{1}^{*}\psi_{1}^{*}(i)\leq \biggl[\sum_{j\in S}\psi_{1}^{*}(j)\pi_{ij}(v_1,\tilde{v}_{2}(i))+c_1(i,v_1,\tilde{v}_{2}(i))\psi_{1}^{*}(i)\biggr].
	  \end{align*}
	  Hence,   \begin{align}
	  \rho_{1}^{*}\psi_{1}^{*}(i)\leq \inf_{v_1\in V_1(i)}\biggl[\sum_{j\in S}\psi_{1}^{*}(j)\pi_{ij}(v_1,\tilde{v}_{2}(i))+c_1(i,v_1,\tilde{v}_{2}(i))\psi_{1}^{*}(i)\biggr].\label{eq 3.10}
	  \end{align}
	  Since $V_1(i)$ is compact, there exist $v_{1,n}^{*}, v_1^{*}\in \mathcal{S}_1$ such that $v_{1,n}^{*}\rightarrow v_1^{*}$ satisfying
	  \begin{align}
	  \rho_{1}^{v_{2,n}}\psi_{1}^{v_{2,n}}(i)=\biggl[\sum_{j\in S}\psi_{1}^{v_{2,n}}(j)\pi_{ij}(v_{1,n}^{*}(i),v_{2,n}(i))+c_1(i,v_{1,n}^{*}(i),v_{2,n}(i))\psi_{1}^{v_{2,n}}(i)\biggr].\label{eq 3.11}
	  \end{align}
	  Now, using Lemma \ref{lemm 3.1}, the dominated convergent theorem and passing  $n\rightarrow\infty$ in (\ref{eq 3.11}), we obtain
	   \begin{align*}
	  \rho_{1}^{*}\psi_{1}^{*}(i)=\biggl[\sum_{j\in S}\psi_{1}^{*}(j)\pi_{ij}(v_1^{*}(i),\tilde{v}_{2}(i))+c_1(i,v_1^{*}(i),\tilde{v}_{2}(i))\psi_{1}^{*}(i)\biggr],
	  \end{align*}
	  Therefore
	   \begin{align}
	  \rho_{1}^{*}\psi_{1}^{*}(i)\geq \inf_{v_1\in V_1(i)}\biggl[\sum_{j\in S}\psi_{1}^{*}(j)\pi_{ij}(v_1,\tilde{v}_{2}(i))+c_1(i,v_1,\tilde{v}_{2}(i))\psi_{1}^{*}(i)\biggr].\label{eq 3.12}
	  \end{align}
	  Hence, from (\ref{eq 3.10}), and (\ref{eq 3.12}), it follows that
	  \begin{align}
	  \rho_{1}^{*}\psi_{1}^{*}(i)=\inf_{v_1\in V_1(i)}\biggl[\sum_{j\in S}\psi_{1}^{*}(j)\pi_{ij}(v_1,\tilde{v}_{2}(i))+c_1(i,v_1,\tilde{v}_{2}(i))\psi_{1}^{*}(i)\biggr].\label{eq 3.13}
	  \end{align}
	  Since $\rho_{1}^{\tilde{v}_{2}}$ is the minimal eigenvalue corresponding to $\tilde{v}_{2}$ of  (\ref{eq 3.13}), we have $ \rho_{1}^{*}\geq \rho_{1}^{\tilde{v}_{2}}$. Suppose $\rho_{1}^{*} > \rho_{1}^{\tilde{v}_{2}}$. Now, from Theorem~\ref{theo 3.1}, for any minimizing  $\hat{v}_1\in {\mathcal{S}}_1$ of (\ref{eq 3.1}),  there exists a finite set $\mathscr{B}_1\supset \mathscr{K}$, such that
			\begin{align}
			\psi_{1}(i)=E^{\hat{v}_1,\tilde{v}_2}_i\bigg[e^{\int_{0}^{\hat{\tau}(\mathscr{B}_1)}(c_1(Y(t),\hat{v}_1(Y(t)),\tilde{v}_2(Y(t)))-\rho_{1}^{\tilde{v}_{2}})dt}\psi_{1}(Y({\hat{\tau}(\mathscr{B}_1)}))\bigg]~\forall i\in \mathscr{B}_1^c,\label{EL3.2A}
			\end{align}
			where $\hat{\tau}(\mathscr{B}_1)=\inf\{t:Y(t)\in \mathscr{B}_1\}=:\tilde{\tau}_1$. Since $\rho_{1}^{*} > \rho_{1}^{\tilde{v}_{2}}$, by similar arguments as in \cite[Lemma~3.4]{BP} we deduce that
			\begin{equation}\label{EL3.2B}
			\psi_{1}^{*}(i)\leq E^{\hat{v}_1,\tilde{v}_2}_i\bigg[e^{\int_{0}^{\hat{\tau}(\mathscr{B}_1)}(c_1(Y(t),\hat{v}_1(Y(t)),\tilde{v}_2(Y(t)))-\rho_{1}^{\tilde{v}_2})dt}\psi_{1}^{*}(Y({\hat{\tau}(\mathscr{B}_1)}))\bigg]~\forall i\in \mathscr{B}_1^c.
			\end{equation} From (\ref{EL3.2A}) and (\ref{EL3.2B}), we obtain
			\begin{equation}\label{EL3.2C}
			(\psi_{1}-\psi_{1}^{*})(i)\geq E^{\hat{v}_1,\tilde{v}_2}_i\bigg[e^{\int_{0}^{\hat{\tau}(\mathscr{B}_1)}(c_1(Y(t),\hat{v}_1(Y(t)),\tilde{v}_2(Y(t)))-\rho_{1}^{\tilde{v}_2})dt}(\psi_{1}-\psi_{1}^{*})(Y({\hat{\tau}(\mathscr{B}_1)}))\bigg]~\forall i\in \mathscr{B}_1^c.
			\end{equation}
Now choosing an appropriate constant $\theta$ (e.g., $\theta = \max_{\mathscr{B}_1}\frac{\psi_{1}}{\psi_{1}^*}$), we have $(\psi_{1}-\theta\psi_{1}^{*}) \geq 0$ in $\mathscr{B}_1$ and for some $\hat{i}_0\in \mathscr{B}_1,$ \,$(\psi_{1}-\theta\psi_{1}^{*})(\hat{i}_0) = 0$. Thus, in view of (\ref{EL3.2C}), we get $(\psi_{1}-\theta\psi_{1}^{*}) \geq 0$ in $S$. Now combining (\ref{eq 3.1}) and (\ref{eq 3.13}), we get
     \begin{align*}
	  \rho_{1}^{\tilde{v}_{2}}(\psi_{1} - \theta\psi_{1}^{*})(\hat{i}_0) \geq \biggl[\sum_{j\in S}(\psi_{1} - \theta\psi_{1}^{*})(j)\pi_{\hat{i}_0 j}(\hat{v}_1(\hat{i}_0),\tilde{v}_{2}(\hat{i}_0)) + c_1(\hat{i}_0,\hat{v}_1(\hat{i}_0),\tilde{v}_{2}(\hat{i}_0))(\psi_{1} - \theta\psi_{1}^{*})(\hat{i}_0)\biggr].	
	  \end{align*} This implies that
	  \begin{equation}\label{EL3.2D}
	  \sum_{j\neq \hat{i}_0}(\psi_{1} - \theta\psi_{1}^{*})(j)\pi_{\hat{i}_0 j}(\hat{v}_1(\hat{i}_0),\tilde{v}_{2}(\hat{i}_0)) = 0\,.
	  \end{equation}
Since, $\{Y(t)\}_{t\geq 0}$ is irreducible under $(\hat{v}_1,\tilde{v}_{2})\in \mathcal{S}_1 \times \mathcal{S}_2$, from (\ref{EL3.2D}) it follows that $\psi_{1} \equiv \theta\psi_{1}^{*}$. But this is a contradiction to the fact that $\rho_{1}^{*} > \rho_{1}^{\tilde{v}_{2}}$. Hence, we deduce that $\rho_{1}^{*} = \rho_{1}^{\tilde{v}_{2}}$.
	  This proves the continuty of the map $\hat{v}_2\rightarrow\rho_{1}^{\hat{v}_2}$. Since $\psi_{1}^{\hat{v}_{2,n}}(i_0)=1$ for all $n\geq 1$, we have $\psi_{1}^{*}(i_0)=1$. Hence by Theorem \ref{theo 3.1}, we have $\psi_{1}^{*}$ is the unique solution of (\ref{eq 3.1}). Thus $\psi_{1}^{*}=\psi_1^{\tilde{v}_{2}}$. This proves the continuity of the map $\hat{v}_2\rightarrow\psi
	  _{1}^{\hat{v}_2}$. Continuity of  other maps follows by the similar argument.
	\end{proof}

	 Fix $\hat{v}_2\in {\mathcal{S}}_2$. For each $i\in S$, $v_1\in V_1(i)$, set
\begin{align*}
\tilde{F}_1(i,v_1,\hat{v}_2(i))= \bigg[\sum_{j\in S}{\psi}_{1}^{\hat{v}_2}(j)\pi_{ij}(v_1,\hat{v}_2(i))+c_1(i,v_1,\hat{v}_2(i)){\psi}_{1}^{\hat{v}_2}(i)\bigg],
\end{align*}
where $\psi_{1}^{\hat{v}_2}$ is the solution of (\ref{eq 3.1}) corresponding to the strategy $\hat{v}_2\in {\mathcal{S}_2}$.
Let
\begin{align*}
\tilde{H}(\hat{v}_2)=\biggl\{\hat{v}^{*}_1\in {\mathcal{S}_1}:\tilde{F}_1(i,\hat{v}^{*}_1(i),\hat{v}_2(i))=\inf_{v_1\in V_1(i)} \tilde{F}_1(i,v_1,\hat{v}_2(i))~\forall~i\in S\biggr\}.
\end{align*}
Then by the compactness of each $V_1(i)$, it follows that $\tilde{H}(\hat{v}_2)$ is a non empty subset of ${\mathcal{S}}_1$. It is obvious that, $\tilde{H}(\hat{v}_2)$ is convex and closed. Since ${\mathcal{S}}_1$ is compact, $\tilde{H}(\hat{v}_2)$ is also compact.
Similarly, for $i\in S$, $\hat{v}_1\in {\mathcal{S}}_1$, $v_2\in V_2(i)$, we set
\begin{align*}
\tilde{F}_2(i,\hat{v}_1(i),v_2)= \bigg[\sum_{j\in S}{\psi}_{2}^{\hat{v}_1}(j)\pi_{ij}(\hat{v}_1(i),v_2)+c_2(i,\hat{v}_1(i),v_2){\psi}_{2}^{\hat{v}_1}(i)\bigg],~i\in S,
\end{align*}
where $\psi_{2}^{\hat{v}_1}$ is the solution of  (\ref{eq 3.4}) corresponding to the strategy $\hat{v}_1\in {\mathcal{S}}_1$.
 Let
\begin{align*}
\tilde{H}(\hat{v}_1)=\biggl\{\hat{v}^{*}_2\in {\mathcal{S}_2}:\tilde{F}_2(i,\hat{v}_1(i),\hat{v}^{*}_2(i))=\inf_{v_2\in V_2(i)} \tilde{F}_2(i,\hat{v}_1(i),v_2)~\forall~i\in S\biggr\}.
\end{align*}
Then by analogous arguments, $	\tilde{H}(\hat{v}_1)$ is nonempty, convex and is a compact subset of $ {\mathcal{S}}_2$.
Next  set
\begin{align*}
\tilde{H}(\hat{v}_1,\hat{v}_2)=\tilde{H}(\hat{v}_2)\times \tilde{H}(\hat{v}_1).
\end{align*}
From the above argument it is clear that $\tilde{H}(\hat{v}_1,\hat{v}_2)$ is nonempty, convex, and is a compact subset of ${\mathcal{S}}_1\times {\mathcal{S}}_2$. Therefore we may define a map from ${\mathcal{S}}_1\times{\mathcal{S}}_2\rightarrow 2^{{\mathcal{S}}_1}\times 2^{ {\mathcal{S}}_2}$.
\subsection{The existence of Nash equilibria}
 Next lemma proves upper-semicontinuity of certain set valued map. This result will be useful in establishing existence of a Nash equilibrium in the space of stationary Markov strategies.
\begin{lemma}\label{lemm 4.1}
	Let Assumptions \ref{assm 3.1}, \ref{assm 3.2}, and \ref{assm 3.3} hold. Then the map $(\hat{v}_1,\hat{v}_2)\rightarrow\tilde{H}(\hat{v}_1,\hat{v}_2)$ from ${\mathcal{S}}_1\times {\mathcal{S}}_2\rightarrow 2^{{\mathcal{S}}_1}\times 2^{ {\mathcal{S}}_2}$ is upper semicontinuous.
\end{lemma}
\begin{proof}
	Let $\{(v_1^m,v_2^m)\}\in {\mathcal{S}}_1\times {\mathcal{S}}_2$ and $(v_1^m,v_2^m)\rightarrow (\hat{v}_1,\hat{v}_2)$ in ${\mathcal{S}}_1\times {\mathcal{S}}_2$, i.e., for each $i \in S, \, (v_1^m(i),v_2^m(i)) \to (\hat{v_1}(i),\hat{v}_2(i)) \, \, \mbox{in} \, \, V_1(i)\times V_2(i) $. Let $\overline{v}_1^m\in \tilde{H}({v}_2^m)$. Then $\{\overline{v}_1^m\}\subset {\mathcal{S}}_1$. Since ${\mathcal{S}}_1$ is compact, it has a convergent subsequence (denoted by the same sequence by an abuse of notation), such that
	\begin{align*}
\overline{v}_1^m\rightarrow\overline{v}_1 ~\text{in}~ {\mathcal{S}}_1.
	\end{align*}
	Then $(\overline{v}_1^m,v_2^m)\rightarrow (\overline{v}_1,\hat{v}_2)$ in ${\mathcal{S}}_1\times {\mathcal{S}}_2$.
	Note that
	\begin{align*}
	\sum_{j\neq i}{\pi}_{ij}(\overline{v}_1^m,v_2^{m}(i)){\psi}^{{v}_2^{m}}_1(j) \leq \sum_{j\neq i}{\pi}_{ij}(\overline{v}_1^m,v_2^{m}(i))W(j).
	\end{align*}
	Thus from [\cite{HL}, Lemma~8.3.7],  Assumption \ref{assm 3.3} and the (product) topology of ${\mathcal{S}}_k$, $k=1,2$, it follows that for each $i\in S$,
	$$\sum_{j\in S}{\pi}_{ij}(\overline{v}_1^m,v_2^m(i)){\psi}^{v_2^m}_{1}(j)+ c_1(i,\overline{v}_1^m,v_2^m(i)){\psi}^{{v}^m_2}_{1}(i)$$ converges  to
	$$\sum_{j\in S}{\pi}_{ij}(\overline{v}_1,\hat{v}_2(i)){\psi}^{\hat{v}_2}_1(j)+ c_1(i,\overline{v}_1,\hat{v}_2(i)){\psi}^{\hat{v}_2}_1(i).$$
	 Hence we have
	\begin{align}
	\lim_{m\rightarrow \infty}\tilde{F}_1(i,\overline{v}_1^m(i),v^m_2(i))=\tilde{F}_1(i,\overline{v}_1(i),\hat{v}_2(i)).\label{eq 4.1}
	\end{align}
Now fix $\tilde{v}_1\in {\mathcal{S}}_1$ and consider the sequence $(\tilde{v}_1,v_2^m)$. Using the analogous arguments as above, we conclude that
\begin{align}
\lim_{m\rightarrow \infty}\tilde{F}_1(i,\tilde{v}_1(i),v^m_2(i))=\tilde{F}_1(i,\tilde{v}_1(i),\hat{v}_2(i)).\label{eq 4.2}
\end{align}
Since $\overline{v}_1^m\in H(v_2^m)$, for any $m$ we have
\begin{align*}
\tilde{F}_1(i,\tilde{v}_1(i),v^m_2(i))\geq \tilde{F}_1(i,\overline{v}_1^m(i),v^m_2(i)).
\end{align*}
	Thus, in view of (\ref{eq 4.1}) and (\ref{eq 4.2}), taking $m\rightarrow\infty$ in the above equation, for any $\tilde{v}_1\in {\mathcal{S}}_1$ we get
	\begin{align*}
\tilde{F}_1(i,\tilde{v}_1(i),\hat{v}_2(i))\geq \tilde{F}_1(i,\overline{v}_1(i),\hat{v}_2(i)).
	\end{align*}
	Therefore, $\overline{v}_1\in \tilde{H}(\hat{v}_2)$. Suppose $\overline{v}_2^m\in \tilde{H}(v_1^m) $ and along a subsequence $\overline{v}_2^m\rightarrow\overline{v}_2$ in $ {\mathcal{S}}_2$. Then, by the similar arguments as above one can show that $\overline{v}_2\in \tilde{H}(\hat{v}_1)$. This proves that the map $(\hat{v}_1,\hat{v}_2)\rightarrow\tilde{H}(\hat{v}_1,\hat{v}_2)$ is upper-semicontinuous.
\end{proof}
\begin{thm}\label{theo 4.1}
	 Grant Assumptions \ref{assm 3.1}, \ref{assm 3.2}, and \ref{assm 3.3}. Then there exists a Nash equilibrium in the space of stationary Markov strategies ${\mathcal{S}}_1\times {\mathcal{S}}_2.$
	 \begin{proof}
	 	By Lemma \ref{lemm 4.1} and Fan's fixed point theorem \cite{F1}, there exists a fixed point $(\hat{v}^{*}_1,\hat{v}^{*}_2)\in{\mathcal{S}}_1\times {\mathcal{S}}_2$, for the map $(\hat{v}_1,\hat{v}_2)\rightarrow\tilde{H}(\hat{v}_1,\hat{v}_2)$ from ${\mathcal{S}}_1\times {\mathcal{S}}_2\rightarrow 2^{{\mathcal{S}}_1}\times 2^{ {\mathcal{S}}_2}$, i.e.,
	 	\begin{align*}
	(\hat{v}^{*}_1,\hat{v}^{*}_2)\in \tilde{H}(\hat{v}^{*}_1,\hat{v}^{*}_2).
	 	\end{align*}
	 	This implies that $(\rho_{1}^{\hat{v}^{*}_2},\psi_{1}^{\hat{v}^{*}_2})$, $(\rho_{2}^{\hat{v}^{*}_1},\psi_{2}^{\hat{v}^{*}_1})$ satisfy the following coupled HJB equations:
	 	\begin{align}
	 \left\{ \begin{array}{lll}\rho_{1}^{\hat{v}^{*}_2}\psi_{1}^{\hat{v}^{*}_2}(i)&=\displaystyle { \inf_{v_1\in V_1(i)}\biggl[\sum_{j\in S}{\pi}_{ij}(v_1,\hat{v}^{*}_2(i)){\psi}^{\hat{v}^{*}_2}_1(j)+ c_1(i,v_1,\hat{v}^{*}_2(i)){\psi}^{\hat{v}^{*}_2}_1(i)\biggr]}\\
	 &=\biggl[\sum_{j\in S}{\pi}_{ij}(\hat{v}^{*}_1(i),\hat{v}^{*}_2(i)){\psi}^{\hat{v}^{*}_2}_{1}(j)+ c_1(i,\hat{v}^{*}_1(i),\hat{v}^{*}_2(i)){\psi}^{\hat{v}^{*}_2}_1(i)\biggr],\\\psi_{1}^{\hat{v}^{*}_2}(i_0)=1\label{eq 4.3}
	 \end{array}\right.
	 \end{align}
	 and
		\begin{align}
	\left\{ \begin{array}{lll}\rho_{2}^{\hat{v}^{*}_1}\psi_{2}^{\hat{v}^{*}_1}(i)&=\displaystyle { \inf_{v_2\in V_2(i)}\biggl[\sum_{j\in S}{\pi}_{ij}(\hat{v}^{*}_1(i),v_2){\psi}^{\hat{v}^{*}_1}_2(j)+ c_2(i,\hat{v}^{*}_1(i),v_2){\psi}^{\hat{v}^{*}_1}_2(i)\biggr]}\\
	&=\biggl[\sum_{j\in S}{\pi}_{ij}(\hat{v}^{*}_1(i),\hat{v}^{*}_2(i)){\psi}^{\hat{v}^{*}_1}_2(j)+ c_2(i,\hat{v}^{*}_1(i),\hat{v}^{*}_2(i)){\psi}^{\hat{v}^{*}_1}_2(i)\biggr],\\\psi_{2}^{\hat{v}^{*}_2}(i_0)=1.\label{eq 4.4}
	\end{array}\right.
	\end{align}
	 	Now by Theorem \ref{theo 3.1}, from (\ref{eq 4.3}), it follows that
	 	\begin{align}
	 	\rho_1^{\hat{v}^{*}_2}=&\inf_{{v}_1\in\mathscr{A}_1}\rho_1^{{v}_1,\hat{v}^{*}_2}\nonumber\\
	 	&=\rho_1^{\hat{v}^{*}_1,\hat{v}^{*}_2}.\label{eq 4.5}
	 	\end{align}
	 	Similarly, from (\ref{eq 4.4}), we have
	 		\begin{align}
	 	\rho^{\hat{v}^{*}_1}_{2}=&\inf_{{v}_2\in\mathscr{A}_2}\rho_2^{\hat{v}^{*}_1,{v}_2}\nonumber\\
	 	&=\rho_2^{\hat{v}^{*}_1,\hat{v}^{*}_2}\label{eq 4.6}.
	 	\end{align}
	 	
	 	Thus, from equations (\ref{eq 4.5}) and (\ref{eq 4.6}), we get
	 		\begin{align*}
\rho_1^{{v}_1,\hat{v}^{*}_2}\geq \rho_1^{\hat{v}^{*}_1,\hat{v}^{*}_2},~\forall~{v}_1\in \mathscr{A}_1,
	 	\end{align*}
	 	\begin{align*}
\rho_2^{\hat{v}^{*}_1,{v}_2}\geq \rho_2^{\hat{v}^{*}_1,\hat{v}^{*}_2}, ~\forall~{v}_2\in \mathscr{A}_2.
	 	\end{align*}
	 	Hence $(\hat{v}^{*}_1,\hat{v}^{*}_2)\in {\mathcal{S}}_1\times {\mathcal{S}}_2$ is a Nash equilibrium. This completes the proof.
	 	\end{proof}
\end{thm}

Now we prove a converse of Theorem \ref{theo 4.1}.
\begin{thm}\label{theo 4.2}
  Let Assumptions \ref{assm 3.1}, \ref{assm 3.2}, and \ref{assm 3.3} hold. If $(\underline{v}_1^*,\underline{v}_2^*)\in \mathcal{S}_1\times\mathcal{S}_2$ is a Nash equilibrium,
i.e.,
	\begin{align*}
\rho_1^{{v}_1,\underline{v}_2^*}\geq \rho_1^{\underline{v}_1^*,\underline{v}_2^*},~\forall~{v}_1\in \mathscr{A}_1,
\end{align*}
\begin{align*}
\rho_2^{\underline{v}_1^*,{v}_2}\geq \rho_2^{\underline{v}_1^*,\underline{v}_2^*}, ~\forall~{v}_2\in \mathscr{A}_2.
\end{align*}
Then $\underline{v}_1^*\in \mathcal{S}_1$ is a minimizing selector of (\ref{eq 3.1}) (corresponding to fixed strategy $\underline{v}_2^*\in\mathcal{S}_2$ of player 2) and $\underline{v}_2^*\in \mathcal{S}_2$ is a minimizing selector of (\ref{eq 3.4}) (corresponding to fixed strategy $\underline{v}_1^*\in\mathcal{S}_1$ of player 1).
\end{thm}
\begin{proof}
	 Applying analogous arguments as in [\cite{BP}, Lemma 3.4 and Remark 3.1], one can prove that for the given pair $(\underline{v}_1^*,\underline{v}_2^*)\in \mathcal{S}_1\times\mathcal{S}_2$, there exists a eigenpair
	 $(\rho_{1}^{\underline{v}_1^*,\underline{v}_2^*},\psi_{1}^{\underline{v}_1^*,\underline{v}_2^*})\in \mathbb{R}\times L^\infty_{W}$, $\psi_{1}^{\underline{v}_1^*}>0$ and $\rho_{1}^{\underline{v}_1^*,\underline{v}_2^*}\geq 0$ satisfying
	 	
		\begin{align}
	\left\{ \begin{array}{lll}&\rho_{1}^{\underline{v}_1^*,\underline{v}_2^*}\psi_{1}^{\underline{v}_1^*,\underline{v}_2^*}(i)=\sum_{j\in S}\pi_{ij}(\underline{v}_1^*(i),\underline{v}_2^*(i))\psi_{1}^{\underline{v}_1^*,\underline{v}_2^*}(j)+c_1(i,\underline{v}_1^*(i),\underline{v}_2^*(i))\psi_{1}^{\underline{v}_1^*,\underline{v}_2^*}(i),\\&\psi_{1}^{\underline{v}_1^*,\underline{v}_2^*}(i_0)=1.\label{eq 4.7}
	\end{array}\right.
	\end{align}
	Also, for given $\underline{v}_2^*\in \mathcal{S}_2$, there exists a minimal eigenpair $(\rho_{1}^{\underline{v}_2^*},\psi_{1}^{\underline{v}_2^*})\in \mathbb{R}_+\times L^\infty_{W}$, $\psi_{1}^{\underline{v}_2^*}>0$, satisfying
		\begin{align}
	\left\{ \begin{array}{lll}&\rho_{1}^{\underline{v}_2^*}\psi_{1}^{\underline{v}_2^*}(i)=\displaystyle { \inf_{v_1\in V_1(i)}\biggl[\sum_{j\in S}\pi_{ij}(v_1,\underline{v}_2^*(i))\psi_{1}^{\underline{v}_2^*}(j)+c_1(i,{v}_1,\underline{v}_2^*(i))\psi_{1}^{\underline{v}_2^*}(i)\biggr]},\\&\psi_{1}^{\underline{v}_2^*}(i_0)=1.\label{eq 4.8}
	\end{array}\right.
	\end{align}
 Since $\rho_{1}^{\underline{v}_2^*}$ is a minimal eigenvalue of (\ref{eq 4.8}), corresponding to $\underline{v}_2^*$, we have
	\begin{align}
\rho_{1}^{\underline{v}_2^*}=\inf_{{v}_1\in\mathscr{A}_1}\rho_{1}^{v_1,\underline{v}_2^*}.\label{eq 4.9}
	\end{align}
	Also, we have
	\begin{align*}
	\rho_1^{{v}_1,\underline{v}_2^*}\geq \rho_1^{\underline{v}_1^*,\underline{v}_2^*},~\forall~{v}_1\in \mathscr{A}_1.
	\end{align*}
	Hence,
		\begin{align}
	\inf_{{v}_1\in\mathscr{A}_1}\rho_1^{{v}_1,\underline{v}_2^*}\geq \rho_1^{\underline{v}_1^*,\underline{v}_2^*}.\label{eq 4.10}
	\end{align}
	So, by (\ref{eq 4.9}) and (\ref{eq 4.10}), we obtain
		\begin{align*}
	\rho^{\underline{v}_2^*}_1\geq \rho^{\underline{v}_1^*,\underline{v}_2^*}_1.
	\end{align*}
	Also, from (\ref{eq 4.9}), we have
	\begin{align*}
	\rho_{1}^{\underline{v}_2^*}\leq \rho_{1}^{\underline{v}_1^*,\underline{v}_2^*}.
	\end{align*}
		Hence, we deduce that
	\begin{align}
	\rho^{\underline{v}_2^*}_1=\rho^{\underline{v}_1^*,\underline{v}_2^*}_1.\label{eq 4.11}
	\end{align}

Now, applying Ito-Dynkin formula, from (\ref{eq 4.7}), it follows that
	\begin{align*}
\psi_{1}^{\underline{v}_1^*,\underline{v}_2^*}(i)=E^{\underline{v}_1^*,\underline{v}_2^*}_i\bigg[e^{\int_{0}^{T\wedge\hat{\tau}(\mathscr{B}_1)}(c_1(Y(t),\underline{v}_1^*(Y(t)),\underline{v}_2^*(Y(t)))-\rho_{1}^{\underline{v}_1^*,\underline{v}_2^*})dt}\psi_{1}^{\underline{v}_1^*,\underline{v}_2^*}(Y({T\wedge\hat{\tau}(\mathscr{B}_1)}))\bigg]~\forall i\in \mathscr{B}_1^c,
\end{align*}
where $\mathscr{B}_1$ is as in Theorem \ref{theo 3.1}.
Now, by Fatou's Lemma, taking $T\rightarrow\infty$ in the above equation, we get
	\begin{align}
\psi_{1}^{\underline{v}_1^*,\underline{v}_2^*}(i)\geq E^{\underline{v}_1^*,\underline{v}_2^*}_i\bigg[e^{\int_{0}^{\hat{\tau}(\mathscr{B}_1)}(c_1(Y(t),\underline{v}_1^*(Y(t)),\underline{v}_2^*(Y(t)))-\rho_{1}^{\underline{v}_1^*,\underline{v}_2^*})dt}\psi_{1}^{\underline{v}_1^*,\underline{v}_2^*}(Y({\hat{\tau}(\mathscr{B}_1)}))\bigg]~\forall i\in \mathscr{B}_1^c.\label{eq 4.12}
\end{align}
Again, using (\ref{eq 4.8}), from Theorem \ref{theo 3.1}, it follows that
	\begin{align}
\psi_{1}^{\underline{v}_2^*}(i)\leq E^{\underline{v}_1^*,\underline{v}_2^*}_i\bigg[e^{\int_{0}^{\hat{\tau}(\mathscr{B}_1)}(c_1(Y(t),\underline{v}_1^*(Y(t)),\underline{v}_2^*(Y(t)))-\rho_{1}^{\underline{v}_2^*})dt}\psi_{1}^{\underline{v}_2^*}(Y({\hat{\tau}(\mathscr{B}_1)}))\bigg]~\forall i\in \mathscr{B}_1^c.\label{eq 4.13}
\end{align}
So, by (\ref{eq 4.12}) and (\ref{eq 4.13}), we obtain
	\begin{align}
\psi_{1}^{\underline{v}_1^*,\underline{v}_2^*}(i)&-\psi_{1}^{\underline{v}_2^*}(i)\nonumber\\
&\geq E^{\underline{v}_1^*,\underline{v}_2^*}_i\bigg[e^{\int_{0}^{\hat{\tau}(\mathscr{B}_1)}(c_1(Y(t),\underline{v}_1^*(Y(t)),\underline{v}_2^*(Y(t)))-\rho_{1}^{\underline{v}_2^*})dt}(\psi_{1}^{\underline{v}_1^*,\underline{v}_2^*}-\psi_{1}^{\underline{v}_2^*})(Y({\hat{\tau}(\mathscr{B}_1)}))\bigg]~\forall i\in \mathscr{B}_1^c.\label{eq 4.14}
\end{align}
Now arguing as in the proof of Lemma~\ref{lemm 3.4}, we obtain $\psi_{1}^{\underline{v}_1^*,\underline{v}_2^*}(i) \equiv \psi_{1}^{\underline{v}_2^*}$. Thus, from (\ref{eq 4.7}) and (\ref{eq 4.8}) it follows that $\underline{v}_1^*$ is a minimizing selector of (\ref{eq 3.1}) (for fixed strategy $\underline{v}_2^*\in\mathcal{S}_2$ of player 2). Following  similar arguments one can show that $\underline{v}_2^*$ is a minimizing selector of (\ref{eq 3.4}) (for fixed strategy $\underline{v}_1^*\in\mathcal{S}_1$ of player 1). This completes the proof.
\end{proof}
 \section{ Example}
In this section, we present an illustrative example in wherein  transition rates are unbounded and cost rates are nonnegative and unbounded.
\begin{example}
	Consider a shop which deals with only one type of product for buying and selling. Suppose there are two workers, say, player 1 and player 2 for buying and selling the products, respectively. The number of stocks in the shop is a finite subset of the set of natural numbers $\mathbb{N}$ at each time $t\geq 0$. There are `natural' buying and selling rates, say $\tilde{\mu}$ and $\lambda$, respectively, and buying parameters $h_1$ controlled by player 1 and selling parameters $h_2$ controlled by player 2. When the state of the system is $i\in S:=\{1,2,\cdots\}$ (i.e., number of items in the shop) , player 1 takes an action $u_1$ from a given set $U_1(i)$, which may increase $(h_1(i,u_1)\geq 0)$ or decrease $(h_1(i,u_1)\leq 0)$ the buying rate. These actions produce a payoff denoted by $r_1(i,u_1)$ per unit time. Similarly, if the state is $i\in S$, player 2 takes an action $u_2$ from a set $U_2(i)$ to decrease $(h_2(i,b)\leq 0)$ or to increase $(h_2(i,b)\geq 0)$ the selling rate.  These actions result in a payoff denoted by $r_2(i,u_2)$ per unit time. We assume that when the stock of items in the shop becomes 1, the first player may buy any number of stocks of that item as much as he/she likes depending upon the availability of cash. In addition, we assume that player $k, (k=1,2)$ `gets' a reward $r_k(i):=p_ki$ or incurs a cost $r_k(i):=p_ki$ for each unit of time during which the system remains in the state $i\in S$, where $p_k>0$ is a fixed reward fee, and $p_k<0$, a fixed cost fee, per stock, from the owner.\\
	We next formulate this model as a continuous-time Markov game. The corresponding transition rate $\overline{\pi}_{ij}(u_1,u_2)$ and payoff rate $\bar{c}_k(i,u_1,u_2)$ for player $k (k=1,2)$ are given as follows: for $(1,u_1,u_2)\in K$ ($K$ as in the game model (\ref{eq 2.1})).
	\begin{equation}\label{ExDe1}
	\overline{\pi}_{1j}(u_1,u_2)>0~\forall j\geq 2,~\text{such that}~\sum_{j\in S}\overline{\pi}_{1j}(u_1,u_2)=0,~\text{and}~\overline{\pi}_{1j}(u_1,u_2) \leq e^{-2\theta j}~\forall~j\geq 2,
	\end{equation}
	where $\theta>0$ is a constant.\\
	Also, for $(i,u_1,u_2)\in K$ with $i\geq 2$,
	\begin{align}
	\overline{\pi}_{ij}(u_1,u_2)= \left\{ \begin{array}{lll}\displaystyle{}&\lambda i+h_2(i,u_2),
	~\text{if}~j=i-1\nonumber\\
	&-\tilde{\mu} i-\lambda i-h_1(i,u_1)-h_2(i,u_2),~\text{if}~ j=i\\
	& \tilde{\mu} i+h_1(i,u_1),~\text{if}~j=i+1\nonumber\\
	&0,~\text{otherwise}\displaystyle{}.
	\end{array}\right.
	\end{align}
	\begin{align}
	\bar{c}_1(i,u_1,u_2):=ip_1-r_1(i,u_1),~\bar{c}_2(i,u_1,u_2)=ip_2-r_2(i,u_2)~\text{ for }~(i,u_1,u_2)\in K.\label{eq 5.1}
	\end{align}
	We now investigate conditions under which there exists a Nash-equilibrium. To this end we make following assumptions:
	\begin{enumerate}
		\item [(I)] For each $i\in S$, $U_1(i)=U_2(i)=[0,L]$, $L>0$ is a constant.
		
		\item [(II)] Let $\lambda\geq \tilde{\mu}>0$, $\tilde{\mu}i+h_1(i,u_1)>0$, and $\lambda i+h_2(i,u_2)>0$ for all $(i,u_1,u_2)\in K$ with $i\geq 2$; and  assume that $h_1(1,u_1)>0$ and $h_2(1,u_2)=0$ for all $(u_1,u_2)\in U_1(i)\times U_2(i).$
		
		\item [(III)] The functions $h_1(i,u_1)$, $h_2(i,u_2)$, $r_1(i,u_1)$, $r_2(i,u_2)$, and $\overline{\pi}_{11}(u_1,u_2)$ are continuous in $(u_1,u_2)$ for each fixed $i\in S$. Suppose there exists a finite set $\mathscr{K}$ such that $h_k(i,u_k)=\frac{u_k}{e^{\theta i}}I_{\mathscr{K}}(i)$ and $1\in \mathscr{K}$.
		Also assume that $\inf_{(u_1,u_2)\in U_1(\cdot)\times U_2(\cdot)}r_k(\cdot,u_k)$ is norm like function for $k=1,2$.

		\item [(IV)] 	Suppose $ip_k-r_k(i,u_k)\geq 0~\forall i\in S, (u_1,u_2)\in U_1(i)\times U_2(i)$ and $(1-e^{-\theta})\lambda+(1-e^\theta)\tilde{\mu}>p_k$ for $k=1,2$.
		
		
	\end{enumerate}
\end{example}

\begin{proposition}\label{Prop 6.1}
	Under conditions (I)-(IV), the above controlled system satisfies the Assumptions   \ref{assm 3.1}, \ref{assm 3.2}, and \ref{assm 3.3}. Hence by Theorem \ref{theo 4.1}, there exists a Nash-equilibrium.
\end{proposition}
\begin{proof}
	Take a Lyapunov function as $V(i):=e^{\theta i}$ for $i\in S$ for some $\theta>0$ as described earlier. Then, we have $V(i)\geq 1$ for all $i\in S$. Now for each $i\geq 2$, and $(u_1,u_2)\in U_1(i)\times U_2(i)$, we have
	\begin{align}
	\sum_{j\in S}& \overline{\pi}_{ij}(u_1,u_2)V(j)=\overline{\pi}_{i(i-1)}(u_1,u_2)V(i-1)+V(i)\overline{\pi}_{ii}(u_1,u_2)+V(i+1)\overline{\pi}_{i(i+1)}(u_1,u_2)\nonumber\\
	&=e^{\theta i}\bigg[(\lambda i+h_2(i,u_2))e^{-\theta}-( i\tilde{\mu}+\lambda i +h_1(i,u_1)+h_2(i,u_2))+(\tilde{\mu} i+h_1(i,u_1))e^{\theta}\bigg]\nonumber\\
	&=e^{\theta i} i\bigg[\tilde{\mu} (e^\theta-1)+\lambda(e^{-\theta}-1)+\frac{e^{\theta}h_1(i,u_1)+e^{-\theta}h_2(i,u_2)-h_1(i,u_1)-h_2(i,u_2)}{i}\bigg]\nonumber\\
	&=iV(i)[\tilde{\mu} (e^\theta-1)+\lambda(e^{-\theta}-1)]+\biggl[u_1(e^\theta-1)+u_2(e^{-\theta}-1)\biggr]I_{\mathscr{K}}(i)\nonumber\\
	&	\leq iV(i)[\tilde{\mu} (e^\theta-1)+\lambda(e^{-\theta}-1)]+L(e^\theta-1)I_{\mathscr{K}}(i).\label{eq 5.2}
	\end{align}
	Now for every $\theta>0$, we know
	\begin{align*}
	\lambda(e^{-\theta}-1)+\tilde{\mu}(e^\theta-1)<0\Leftrightarrow \tilde{\mu}<\lambda e^{-\theta}.
	\end{align*}
	Let $[\tilde{\mu}(e^\theta-1)+\lambda(e^{-\theta}-1)]=-\alpha$ for some $\alpha>0$. Also, let $\ell(i)=i\alpha$ and $C_4=\max \biggl\{L(e^\theta -1),\frac{e^{-2\theta}}{1-e^{-\theta}}\biggr\}$ (see (\ref{eq 5.4})).
	Then for $i\geq 2$,
	\begin{align}
	\sup_{(u_1,u_2)\in U_1(i)\times U_2(i)}\sum_{j\in S}V(j)\overline{\pi}_{ij}(u_1,u_2)\leq C_4 I_{\mathscr{K}}(i)-\ell(i)V(i) ~\forall i\in S.\label{eq 5.3}
	\end{align}
	Also, we have
	\begin{align}
	\sum_{j\in S}\overline{\pi}_{1j}(u_1,u_2)V(j) < \bar{\pi}_{11}(u_1,u_2)e^{\theta} + \sum_{j\geq 2}e^{-2\theta j}e^{\theta j}\leq \bar{\pi}_{11}(u_1,u_2)e^{\theta} + \frac{e^{-2\theta}}{1 - e^{-\theta}} <\infty.\label{eq 5.4}
	\end{align}	
	Since $-\ell(i)<1$ for all $i\in S$.
	Hence from (\ref{eq 5.3}) and (\ref{eq 5.4}), for $i\geq 1$, we have
	\begin{align}
	\sum_{j\in S}\overline{\pi}_{ij}(u_1,u_2)V(j)\leq C_1 V(i)+C_2, ~\text{ where}~C_1=1~\text{ and}~  C_2=C_4.\label{eq 5.6}
	\end{align}

	For $i\geq 2$,
	\begin{align}
	-\overline{\pi}_{ii}(u_1,u_2)&=\tilde{\mu} i+\lambda i+h_1(i,u_1)+h_2(i,u_2)\nonumber\\
	&\leq i(\tilde{\mu}+\lambda)+2L\nonumber\\
	&\leq \frac{1}{\theta}(\tilde{\mu}+\lambda) V(i)+2L V(i)\nonumber\\
	&=(2L+\tilde{\mu}+\lambda)\frac{1}{\theta} V(i)\nonumber\\
	&=C_3 V(i).\label{eq 5.7}
	\end{align}
	Take $W=\tilde{W}=V$.
	Now for $k=1,2$
	\begin{align}
	\ell(i)-\sup_{(u_1,u_2)\in U_1(i)\times U_2(i)}\bar{c}_k(i,u_1,u_2)&=\alpha i-ip_k+\inf_{u_k\in U_k(i)}r_k(i,u_k)\nonumber\\
	&=i\beta_k + \inf_{u_k\in U_k(i)}r_k(i,u_k).\label{eq 5.8}
	\end{align}
	We see that from condition (IV), that $\beta_k=\alpha-p_k\geq 0$.
	So, $\ell(i)-\sup_{(u_1,u_2)\in U_1(i)\times U_2(i)}\bar{c}_k(i,u_1,u_2)$ is norm-like function for $k=1,2$.
	Now by (\ref{eq 5.6}), we say Assumption \ref{assm 3.1} (i) holds.
	Also by (\ref{ExDe1}) and (\ref{eq 5.7}), Assumption \ref{assm 3.1} (ii) is verified.\\
		Now we verify Assumption \ref{assm 3.2}.
	By (\ref{eq 5.3}), (\ref{eq 5.4}) and (\ref{eq 5.8}), it is easy to see that Assumption \ref{assm 3.2} is satisfied. \\
	Now by condition (III) and (\ref{eq 5.1}), we say $c_k(i,u_1,u_2)$ and $\overline{\pi}_{ij}(u_1,u_2)$ are continuous in $(u_1,u_2)\in U_1(i)\times U_2(i)$ for each fixed $i,j\in S$ and for $k=1,2$. So, Assumption \ref{assm 3.3} (i) is verified.
	By (\ref{eq 5.2}) and (\ref{eq 5.4}) and condition (III), we say that Assumption \ref{assm 3.3} (ii) is verified. Also, from (\ref{ExDe1}) it is easy to see that Assumption \ref{assm 3.3} (iii) is satisfied.\\
	Hence by Theorem \ref{theo 4.1} there exists a Nash-equilibrium for this controlled process.
\end{proof}
\section{Acknowledgment}
The research work of Mrinal K. Ghosh is partially supported by UGC Centre for Advanced Study. The research work of Somnath Pradhan is partially supported by a National Postdoctoral Fellowship PDF/2020/001938.

\bibliographystyle{elsarticle-num}

\end{document}